\documentclass[12pt,reqno]{amsart}

\usepackage[text={425pt,650pt},centering]{geometry}

\usepackage{amsfonts,amsmath,amsthm, amssymb}
\usepackage{cases}
\usepackage[usenames]{color}
\usepackage{enumerate} 

\theoremstyle{plain}
\newtheorem{thm}{Theorem}[section]

\newtheorem{prop}[thm]{Proposition}
\theoremstyle{remark}
\newtheorem{rem}{\bf{Remark}}
\numberwithin{equation}{section}

\newcommand{\E}{\mathbb{E}}

\newcommand{\N}{\mathbb{N}}

\newcommand{\R}{\mathbb{R}}

\newcommand{\T}{\mathbb{T}}
\newcommand{\Z}{\mathbb{Z}}

\newcommand{\cA}{\mathcal{A}}
\newcommand{\cB}{\mathcal{B}}

\newcommand{\Lip}{{\rm Lip\,}}

\newcommand{\al}{\alpha}

\newcommand{\del}{\delta}
\newcommand{\ep}{\varepsilon}

\newcommand{\sig}{\sigma}

\newcommand{\Del}{\Delta}
\newcommand{\Gam}{\Gamma}

\newcommand{\Sig}{\Sigma}

\newcommand{\ol}{\overline}

\begin{document}
\title[Nonconvex Mean Field Games]{A Note on Nonconvex Mean Field Games}

\author{Hung Vinh Tran}
\address[Hung V. Tran]
{
Department of Mathematics, 
University of Wisconsin Madison, Van Vleck hall, 480 Lincoln drive, Madison, WI 53706, USA}
\email{hung@math.wisc.edu}

\thanks{
The author is partially supported in part by NSF grant DMS-1615944. 
}

\keywords{Mean Field Games; zero-sum differential games; nonconvex setting; }
\subjclass[2010]{
35B10 
35B20 
35K40  
35D40 
35F21 
}

\date{\today}
\maketitle

\begin{abstract}
We introduce a nonconvex Mean Field Games system by studying a model with
 a large number of identical pairs of players who are all rational, 
and each pair plays an identical zero-sum differential game.
We study existence and uniqueness of solutions for a simple  system in this context.
\end{abstract}

\tableofcontents


\section{Introduction}
\subsection{Heuristic derivation}
Mean Field Games were introduced independently by Caines, Huang, Malham\'e \cite{CHM1, CHM2} and Lasry, Lions \cite{LaLi1, LaLi2, LaLi3} 
to study systems with large numbers of identical agents in competition. 
In the competition, each agent is rational and seek to optimize a value (payoff) functional by choosing appropriate controls. 
The interactions between them are given by a mean field coupling term that aggregates their individual contributions.
We then let the number of agents tend to infinity and take the average to obtain a mean field limit,
 in which we observe the distribution of the agents as a probability measure.
A typical Mean Field Games system looks like
\begin{equation*}
{\rm (MFG)} \quad
\begin{cases}
u_t + H(x,Du) =\ep \Del u + F(x,m) \qquad &\text{ in } \T^n \times (0,T),\\
-m_t - \text{div}(D_pH(x,Du)m) = \ep \Del m \qquad &\text{ in } \T^n \times (0,T),\\
u(x,0)=u_0(x), m(x,T)=m_T(x) \qquad &\text{ on } \T^n.
\end{cases}
\end{equation*}
Here $T>0$, $\ep \geq 0$ are given parameters, and $\T^n=\R^n/\Z^n$ is the $n$-dimensional torus.
The first equation, a Hamilton-Jacobi-Bellman type equation, is forward in time and associated with
an optimal control problem, and the unknown $u=u(x,t)$ is the value (payoff) function of an average agent.
The second equation, a Fokker-Planck equation, is backward in time and the unknown $m=m(x,t)$
describes the density (distribution) of the agents.
For each fixed $t\in [0,T]$, $m(\cdot,t)$ is a probability measure.
In this context, the Hamiltonian $H=H(x,p):\T^n \times \R^n \to \R$ is assumed to be convex in $p$
because of the optimal control framework. The coupling term $F(x,m):\T^n \times \mathcal{P}(\T^n) \to \R$ encodes the interactions between
each agent and the mean field. Here $\mathcal{P}(\T^n)$ is the set of all Radon probability measures on $\T^n$.

In the lecture notes of Cardaliaguet \cite{Car1}, and Gomes, Pimentel, Voskanyan \cite{GPV}, 
the time direction in (MFG) is reversed. To go from this setting to theirs, we simply set $\ol{u}(x,t)=u(x,T-t)$,
$\ol{m}(x,t)=m(x,T-t)$ for all $(x,t) \in \T^n \times [0,T]$. Then $(\ol{u},\ol{m})$ satisfies
\begin{equation}\label{MFG'}
\begin{cases}
-\ol{u}_t + H(x,D\ol{u}) =\ep \Del \ol{u} + F(x,\ol{m}) \qquad &\text{ in } \T^n \times (0,T),\\
\ol{m}_t - \text{div}(D_pH(x,D\ol{u})\ol{m}) = \ep \Del \ol{m} \qquad &\text{ in } \T^n \times (0,T),\\
\ol{u}(x,T)=u_0(x), \ol{m}(x,0)=m_T(x) \qquad &\text{ on } \T^n.
\end{cases}
\end{equation}

A quick and heuristic way in \cite{Car1,GPV} to derive \eqref{MFG'} is the following.
An average agent controls a stochastic differential equation
\[
dX_t = \al_t\,dt +\sqrt{2\ep}dB_t
\]
where $B_t$ is a standard Brownian motion. He/she aims at minimizing the value functional
\[
\E \left[ \int_0^T (L(X_s,\al_s)+F(X_s,\ol{m}(s))\,ds + u_0(X_T) \right].
\]
Here the Lagrangian $L=L(x,q):\T^n \times \R^n \to \R$ is the Legendre transform of $H$.
It is important noting that $F$ plays a role in this minimizing problem.

The value functional of an average agent is then given by the first equation in \eqref{MFG'}.
Heuristically, his/her optimal control is given in a feedback form by $\al^*(x,t) = -D_pH(x,D\ol{u})$.
As all agents are rational, they all move with a velocity which is due to both the diffusion
and the drift term $-D_pH(x,D\ol{u})$, which leads to the second equation in \eqref{MFG'}.

We also refer the readers to the surveys of Gu\'eant, Lasry, Lions \cite{GLL}
and Gomes, Sa\'ude \cite{GS} for further discussions on (MFG) and applications.

\subsection{Nonconvex Mean Field Games}
We give a heuristic derivation here in case $\ep=0$.
For now, let us just assume that $F(x,m):\T^n \times \mathcal{P}(\T^n) \to \R$ is nice enough.

We consider a large number of identical pairs of players who are all rational, 
and each pair plays an identical zero-sum differential game.
In each pair, player I aims at maximizing while player II aims at minimizing a certain payoff functional 
by controlling the dynamics of a particle in $\T^n$, which represents the location of the pair in the game.

Fix $T>0$. Let $A,B$ be two compact metric spaces. 
For $t \in [0,T)$, let
\begin{align*}
&\cA_t:=\left\{a:[t,T]\to A\,:\, a \text{ is measurable}\right\},\\
&\cB_t:=\left\{b:[t,T]\to B\,:\, b \text{ is measurable}\right\},
\end{align*}
be the set of possible controls in time $[t,T]$ of players I and II, respectively.
We henceforth identify any two controls which agree a.e.

Assume that the dynamics is given by an ordinary differential equation
\[
{\rm(ODE)} \quad
\begin{cases}
y_x'(s) = f(y_x(s), a(s),b(s)) \quad \text{ for } s\in (t,T),\\
y_x(t) =x \in \T^n,
\end{cases}
\]
for given controls $a(\cdot) \in \cA_t$ of player I,
and $b(\cdot) \in \cB_t$ of player II.
Here, $f:\T^n \times A\times B \to \T^n$ is a given vector field satisfying: there exists $C>0$ such that
\[
\begin{cases}
f \in C(\T^n \times A \times B),\\
|f(y_1,a,b)-f(y_2,a,b)| \leq C|y_1-y_2| \quad \text{ for all } y_1, y_2 \in \T^n, a\in A, b\in B.
\end{cases}
\]
Under the conditions on $f$, (ODE) has a unique solution.
Associated with (ODE) is the payoff functional
\[
C_{x,t}(a(\cdot),b(\cdot))=\int_t^T (h(y_x(s),a(s),b(s))+F(y_x(s),\ol{m}(s)))\,ds + u_0(x(T)),
\]
where $h:\T^n \times A \times B \to \R$ is a given function satisfying: there exists $C>0$ so that
\[
\begin{cases}
h \in C(\T^n \times A \times B),\\
|h(y_1,a,b)-h(y_2,a,b)| \leq C|y_1-y_2| \quad \text{ for all } y_1, y_2 \in \T^n, a\in A, b\in B.
\end{cases}
\]
The interpretation is that $h$ is the running payoff and $u_0$ is the terminal payoff.
For this generic pair of players, at time $s$, 
their only knowledge of the whole world is the distribution of other agent represented by the density $\ol{m}(s)$.
At location $y_x(s)$ and with the knowledge of the density $\ol{m}(s)$, player I gains a further payoff value $F(y_x(s),\ol{m}(s))$.
Of course, the goal of player I is to maximize the payoff functional $C_{x,t}(a(\cdot),b(\cdot))$.
On the other hand, player II wants to minimize it (or to maximize $-C_{x,t}(a(\cdot),b(\cdot))$).
One way to interpret this situation is that generic player I prefers to be close to other pairs to gain more value,
while generic player II prefer to avoid the crowds.

The set of strategies for player I beginning at time $t$ is
\[
\Sig_t:=\left\{\al:\cB_t \to \cA_t \text{ non-anticipating} \right\},
\]
where non-anticipating means that, for all $b_1(\cdot), b_2(\cdot) \in \cB_t$ and $s \in [t,T]$,
\[
b_1(\cdot)=b_2(\cdot) \text{ on } [t,s) \Rightarrow \al[b_1](\cdot)=\al[b_2](\cdot) \text{ on } [t,s).
\]
Similarly, the set of strategies for player II beginning at time $t$ is
\[
\Gam_t:=\left\{\beta:\cA_t \to \cB_t \text{ non-anticipating}\right\}.
\]
We call
\begin{align*}
&V(x,t):=\inf_{\beta \in \Gam_t} \sup_{a(\cdot) \in \cA_t} C_{x,t}(a(\cdot),\beta[a](\cdot)),\\
&U(x,t):=\sup_{\al \in \Sig_t} \inf_{b(\cdot) \in \cB_t} C_{x,t}(\al[b](\cdot),b(\cdot)),
\end{align*}
the lower value and the upper values of the game, respectively.

Let
\begin{align*}
&H^-(x,t,p):=\min_{a \in A} \max_{b \in B} \left\{ -f(x,a,b)\cdot p -h(x,a,b)-F(x,\ol{m}(t))\right\},\\
&H^+(x,t,p):=\max_{b \in B} \min_{a \in A} \left\{ -f(x,a,b)\cdot p  -h(x,a,b)-F(x,\ol{m}(t))\right\},
\end{align*}
be the lower and upper Hamiltonians of the game, respectively.
It was shown by Evans, Souganidis \cite{EvS} that, $V$ is the viscosity solution to the lower Hamilton-Jacobi-Isaacs equation
\[
\begin{cases}
-V_t + H^-(x,t,DV) = 0 \qquad &\text{ in } \T^n \times (0,T),\\
V(x,T) = u_0(x) \qquad &\text{ on } \T^n,
\end{cases}
\]
and $U$ is the viscosity solution to the upper Hamilton-Jacobi-Isaacs equation
\[
\begin{cases}
-U_t + H^+(x,t,DU) = 0 \qquad &\text{ in } \T^n \times (0,T),\\
U(x,T) = u_0(x) \qquad &\text{ on } \T^n.
\end{cases}
\]
Since $H^-\geq H^+$, we get $V \leq U$ by using the comparison principle.
Assume further that the zero-sum differential game has a value, that is, $H^-=H^+$.
This means we assume that 
\begin{equation}\label{game-value}
H(x,p)=\min_{a\in A} \max_{b\in B}\left\{ -f(x,a,b)\cdot p -h(x,a,b)\right\}=\max_{b \in B} \min_{a \in A} \left\{ -f(x,a,b)\cdot p  -h(x,a,b)\right\}.
\end{equation}
Once \eqref{game-value} holds, then we have
\[
H^-(x,t,p)=H^+(x,t,p)= H(x,p)-F(x,\ol{m}(t)).
\]
Thus, $U=V$ solves the first equation in \eqref{MFG'}.
Heuristically, for $(x,t)\in \T^n \times (0,T)$,
the optimal strategies of the pair is given by $(a^*,b^*)$ such that, 
for $Y_x(s)=y_x(s,a^*(s),b^*(s))$, we have $Y_x'(s) = -D_pH(Y_x(s), D\ol{u}(Y_x(s))$ (see Cardaliaguet \cite{Car0}).
As all players are rational, all pairs move with a velocity due to the drift term $-D_pH(x,D\ol{u})$,
which gives us the second equation in \eqref{MFG'}.
We thus obtain \eqref{MFG'}, hence (MFG), with $H$ not convex in $p$.

\subsection{A simple system - A case study}
Our main focus in this paper is the following system,
which is a simplified version of the full system (MFG),
\begin{equation}\label{MFG}
\begin{cases}
u_t + H(Du) = \Del u + \rho*(\rho*m) \qquad &\text{ in } \T^n \times (0,T),\\
-m_t - \text{div}(DH(Du) m) = \Del m \qquad &\text{ in } \T^n \times (0,T),\\
u(x,0) = u_0(x), m(x,T) = m_T(x) \qquad &\text{ on } \T^n.
\end{cases}
\end{equation}
Here, the coupling term $F(x,m)=\rho*(\rho*m)$ is very simple and is of nonlocal type.

We assume the following conditions
\begin{itemize}
\item[(A1)] The Hamiltonian $H:\R^n \to \R$ is smooth and there exists $c_0>0$ such that
\[
|DH(p)|+|D^2 H(p)| \leq c_0 \qquad \text{ for all } p \in \R^n.
\]

\item[(A2)] The convolution kernel $\rho \in C_c^\infty(\T^n, [0,\infty))$ 
satisfying that $\rho$ is symmetric, that is, $\rho(x) = \rho(-x)$ for all $x\in \T^n$, and $\int_{\T^n} \rho\,dx=1$.

\item[(A3)] $u_0 \in C^2(\T^n)$ and $m_T \in C(\T^n,[0,\infty))$ with $\int_{\T^n} m_T\,dx=1$.
\end{itemize}

\subsection*{Organization of the paper}
Our main goal here is to study existence and uniqueness of solutions to \eqref{MFG}.
In Section \ref{sec:exist}, we prove that there exist solutions to \eqref{MFG}.
In Section \ref{sec:unique}, we show that, under some additional conditions, we have uniqueness results for \eqref{MFG}.

\subsection*{Notations} We use the following notations
\begin{align*}
&C^2_1(\T^n \times [0,T]) =\left\{ u: \T^n \times [0,T] \to \R\,:\, u, Du, D^2u, u_t \in C(\T^n \times [0,T]) \right\},\\
&C([0,T], L^2(\T^n))=\left\{ \mathbf{v}:[0,T] \to L^2(\T^n)\,:\, \mathbf{v} \text{ is continuous}\right\}.
\end{align*}

\section{Existence of solutions} \label{sec:exist}
Let $M = \max_{\T^n} m_T$. Set
\[
X=\left\{ m\in C([0,T], L^2(\T^n))\,:\, 0 \leq m \leq M \ \text{ a.e. on } \T^n \times [0,T]\right\}.
\]
The main result in this section is 
\begin{thm}\label{thm:exist}
Assume that {\rm (A1)--(A3)} hold.
Then \eqref{MFG} has a pair of solution $(u,m)\in C^2_1(\T^n \times [0,T]) \times C([0,T],L^2(\T^n))$.
\end{thm}

\begin{proof}

For each $m \in X$, there exists a unique solution, $U\in C^2_1(\T^n \times [0,T])$, of
\begin{equation}\label{HJB}
\begin{cases}
U_t + H(DU) = \Del U+ \rho*(\rho*m) \qquad &\text{ in } \T^n \times (0,T),\\
U(x,0) = u_0(x) \qquad &\text{ on } \T^n.
\end{cases}
\end{equation}
Thanks to Proposition \ref{prop:apriori} below, we have that 
\[
\|DU\|_{L^\infty(\T^n \times [0,T])} + \|D^2 U\|_{L^\infty(\T^n \times [0,T])} \leq C,
\]
where $C$ depends only on $c_0, M, T, \|u_0\|_{C^2(\T^n)}, \|\rho\|_{C^2(\T^n)}$ and not on $m$.
This also implies that $\|U_t\|_{L^\infty(\T^n \times [0,T])} \leq C$.
Let $\widetilde m$ be the solution to the Fokker-Planck equation
\begin{equation}\label{FP}
\begin{cases}
-\widetilde m_t - \text{div}(DH(DU)\widetilde m) = \Del \widetilde m  \qquad &\text{ in } \T^n \times (0,T),\\
\widetilde m(x,T) = m_T(x) \qquad &\text{ on } \T^n.
\end{cases}
\end{equation}
In light of the maximum principle, $\widetilde m \in X$.

Define the map $\Phi\,:\, X \to X$ as $\Phi(m)=\widetilde m$.

\smallskip

\noindent {\bf Claim 1.} The map $\Phi$ is continuous.

Let $m_k \to m$ in $X$. Then $\rho*(\rho*m_k) \to \rho*(\rho*m)$ uniformly on $\T^n \times [0,T]$.
By stability of viscosity solutions and \eqref{apriori}, we get that $U_k \to U$ uniformly on $\T^n \times [0,T]$,
where $U$ is the solution of \eqref{HJB}.
The a priori estimate \eqref{apriori} yields further that $DU_k(\cdot,t) \to DU(\cdot,t)$ uniformly  on $\T^n$ for each $t\in [0,T]$.

Let $v_k=\Phi(m_k) - \Phi(m)$. Then $v_k$ satisfies
\[
\begin{cases}
-(v_k)_t - \text{div}(DH(U_k) v_k) - \text{div}((DH(DU_k)-DH(DU))\widetilde m)= \Del v_k  \qquad &\text{ in } \T^n \times (0,T),\\
v_k(x,T) = 0 \qquad &\text{ on } \T^n.
\end{cases}
\]
Multiply this PDE by $v_k$ and integrate on $\T^n$ to get
\begin{align*}
&\frac{d}{dt} \int_{\T^n} \frac{-|v_k(x,t)|^2}{2}\,dx\\
=\ & \int_{\T^n} \left(-|Dv_k|^2 - v_k DH(DU_k)\cdot Dv_k - \widetilde m (DH(DU_k)-DH(DU))\cdot Dv_k \right)\,dx\\
\leq \ & C \int_{\T^n} \left(|v_k|^2 + |D(U_k-U)|^2 \right)\,dx.
\end{align*}
We employ Gronwall's inequality to yield further that, for $t \in [0,T]$,
\[
\|v_k(\cdot,t)\|_{L^2(\T^n)}^2 \leq C \|D(U_k-U)\|_{L^2(\T^n \times [0,T])}^2.
\]
Let $k \to \infty$ to conclude that $\Phi$ is continuous.

\smallskip

\noindent {\bf Claim 2.} The set $K = \ol{\Phi(X)}$ is compact.

Fix a sequence $\{m_k\} \subset X$. As $U_k$ satisfies estimate \eqref{apriori} for all $k \in \N$,
we also have that $\|(U_k)_t\|_{L^\infty(\T^n \times [0,T])} \leq C$ for some constant $C>0$ independent of $k$.
We use the Arzel\`a--Ascoli theorem to extract a subsequence $\{U_{k_j}\}$ of $\{U_k\}$ such that
\[
U_{k_j} \to U \quad \text{ uniformly on } \T^n \times [0,T],
\]
for some $U \in \Lip(\T^n \times [0,T])$.
The estimate \eqref{apriori} gives further that, for each $t\in [0,T]$,
\[
DU_{k_j}(\cdot,t) \to DU(\cdot,t)  \quad \text{ uniformly on } \T^n.
\]
Repeat the argument in Claim 1 to deduce that $\{\Phi(m_{k_j})\}$ is a Cauchy sequence in $X$.
Therefore, $K$ is compact.

We use Claims 1,2 and Schauder's fixed point theorem to conclude that, there exists $m\in K$ such that
$\Phi(m)=m$.
\end{proof}

\begin{prop} \label{prop:apriori}
Assume that {\rm (A1)--(A3)} hold.
Let $U$ be the solution to \eqref{HJB} with $m \in X$ given.
There exists $C>0$ depends only on $c_0, M, T, \|u_0\|_{C^2(\T^n)}, \|\rho\|_{C^2(\T^n)}$ such that
\begin{equation}\label{apriori}
\|DU\|_{L^\infty(\T^n \times [0,T])} + \|D^2 U\|_{L^\infty(\T^n \times [0,T])} \leq C.
\end{equation}
\end{prop}

\begin{proof}
Let $f=\rho*(\rho*m)$. It is straightforward to see that 
\[
\|Df\|_{L^\infty(\T^n \times [0,T])} \leq M \|D\rho\|_{L^\infty(\T^n)}
\quad \text{and} \quad 
\|D^2f\|_{L^\infty(\T^n \times [0,T])} \leq M \|D^2\rho\|_{L^\infty(\T^n)}.
\]
Fix $(x_0,t_0) \in \T^n \times (0,T]$. We use the nonlinear adjoint method
to prove \eqref{apriori}. See Evans \cite{Ev}, Tran \cite{Tr}, Cagnetti, Gomes, Mitake, Tran \cite{CGMT}, Gomes, Pimentel, Voskanyan \cite{GPV},
Mitake, Tran \cite{MT-LN} and the references therein 
for the development of this method.

Consider the adjoint equation to the linearized operator of \eqref{HJB}:
\begin{equation}\label{adj}
\begin{cases}
-\sig_t  -\text{div}(DH(DU)\sig) = \Del \sig \qquad &\text{ in } \T^n \times (0,t_0),\\
\sig(x,t_0) = \del_{x_0} \qquad &\text{ on } \T^n.
\end{cases}
\end{equation}
It is clear that $\sig>0$ in $\T^n \times (0,t_0)$ and $\int_{\T^n} \sig(x,t)\,dx=1$ for all $t\in [0,t_0]$.
Differentiate \eqref{HJB} with respect to $x_i$, multiply by $\sig$ and integrate to yield that
\[
U_{x_i}(x_0,t_0) = \int_0^{t_0} \int_{\T^n} f_{x_i} \sig\,dxdt + \int_{\T^n} (u_0)_{x_i} \sig \,dx.
\]
Hence,
\begin{equation}\label{bdd-1}
|U_{x_i}(x_0,t_0)| \leq t_0 M \|D\rho\|_{L^\infty(\T^n)} + \|Du_0\|_{L^\infty} \leq TM \|D\rho\|_{L^\infty(\T^n)} + \|Du_0\|_{L^\infty}.
\end{equation}

Let $\phi=\frac{|DU|^2}{2}$. 
Differentiate \eqref{HJB} with respect to $x_i$, multiply by $U_{x_i}$ and sum over $i$ to get
\[
\phi_t + DH(DU)\cdot D\phi - Df\cdot DU=\Del \phi - |D^2U|^2.
\]
Multiply the above by $\sig$ and integrate to imply
\begin{equation}\label{bdd-2}
\int_0^{t_0} \int_{\T^n} |D^2U|^2 \sig\,dxdt = \int_0^{t_0} \int_{\T^n} (Df\cdot DU)\sig\,dxdt + \int_{\T^n} \phi(x,0)\sig\,dx - \phi(x_0,t_0) \leq C.
\end{equation}

Next, we differentiate \eqref{HJB} with respect to $x_i$ then $x_j$,
\[
(U_{x_i x_j})_t + DH(DU)\cdot DU_{x_i x_j} + H_{p_k p_l} U_{x_k x_i} U_{x_l x_j} = \Del U_{x_i x_j} + f_{x_i x_j}.
\]
Multiply this identity by $\sig$, integrate and use (A1), (A3), \eqref{bdd-2} to conclude that
\begin{equation}\label{bdd-3}
|U_{x_i x_j}(x_0,t_0)| \leq \int_0^{t_0} \int_{\T^n} (c_0 |D^2 U|^2 + |f_{x_i x_j}|) \sig\,dxdt +\int_{\T^n} |(U_0)_{x_i x_j}| \sig\,dx \leq C.
\end{equation}
\end{proof}

\begin{rem}
The arguments in the proof of the existence result (Theorem \ref{thm:exist}) are quite standard and not new.
A similar form of the proof already appeared in \cite[Section 10.2]{GPV}.
In \cite{GPV}, Gomes, Pimentel and Voskanyan used
the convexity of $H$ to achieve the uniform semiconcavity estimate
of $U_k$, which was then used  to get Claim 2 (the compactness of $K=\ol{\Phi(X)}$).
The main difference here is that we do not require convexity of $H$,
and estimate \eqref{apriori} is obtained thanks to the appearance of the diffusion term and the nonlinear adjoint method.
\end{rem}


\section{Uniqueness of solutions} \label{sec:unique}
We obtain uniqueness of solutions to \eqref{MFG} in this section.
As $H$ is not necessarily convex, it is much harder to perform this task.
We add the following assumption
\begin{itemize}
\item[(A4)] The constant $c_0$, which appears in {\rm (A1)}, satisfies
\[
c_0 < \frac{1}{12M},
\]
where $M = \max_{\T^n} m_T$.
\end{itemize}
Note first that $M \geq 1$ as $m_T \geq 0$ and $\int_{\T^n} m_T(x)\,dx=1$.
Hence,
\begin{equation}\label{eq:c0}
c_0 < \min\left\{\frac{1}{4(M+2)},  \frac{1}{2\sqrt{5}}\right\}.
\end{equation}
Assumption (A4) is like a smallness condition, which is quite restrictive but nevertheless quantitative. 
Note further that  the smallness of $c_0$ does not depend on $T$, and thus, there is no restriction on $T>0$.
See Ambrose \cite{Am} for related results.

Here is one of the main results in this section.
\begin{thm} \label{thm:unique}
Assume that {\rm (A1)--(A4)} hold.
Then \eqref{MFG} has at most one pair of solution  $(u,m)\in C^2_1(\T^n \times [0,T]) \times C([0,T],L^2(\T^n))$.
\end{thm}

\begin{proof}
Let $(u^1,m^1)$ and $(u^2, m^2)$ be two pairs of solutions in  $C^2_1(\T^n \times [0,T]) \times C([0,T],L^2(\T^n))$ to \eqref{MFG}:
\begin{equation}\label{MFG1}
\begin{cases}
u^1_t + H(Du^1) = \Del u^1 + \rho*(\rho*m^1) \qquad &\text{ in } \T^n \times (0,T),\\
-m^1_t - \text{div}(DH(Du^1) m^1) = \Del m^1 \qquad &\text{ in } \T^n \times (0,T),\\
u^1(x,0) = u_0(x), m^1(x,T) = m_T(x) \qquad &\text{ on } \T^n,
\end{cases}
\end{equation}
and
\begin{equation}\label{MFG2}
\begin{cases}
u^2_t + H(Du^2) = \Del u^2 + \rho*(\rho*m^2) \qquad &\text{ in } \T^n \times (0,T),\\
-m^2_t - \text{div}(DH(Du^2) m^2) = \Del m^2 \qquad &\text{ in } \T^n \times (0,T),\\
u^2(x,0) = u_0(x), m^2(x,T) = m_T(x) \qquad &\text{ on } \T^n.
\end{cases}
\end{equation}

Take the difference of first equations of \eqref{MFG1} and \eqref{MFG2} and use (A1) to get
\[
(u^1-u^2)_t + DH(Du^2)\cdot D(u^1-u^2) - c_0 |D(u^1-u^2)|^2 \leq \Del(u^1-u^2) +\rho*(\rho*(m^1-m^2)).
\]
Multiply this by $m^2$ and integrate on $\T^n$ to yield
\[
\frac{d}{dt} \int_{\T^n} (u^1-u^2) m^2\,dx \leq \int_{\T^n} \rho*(\rho*(m^1-m^2)) m^2\,dx + c_0 \int_{\T^n} |D(u^1-u^2)|^2 m^2\,dx.
\]
A similar computation gives
\[
\frac{d}{dt} \int_{\T^n} (u^2-u^1) m^1\,dx \leq \int_{\T^n} \rho*(\rho*(m^2-m^1)) m^1\,dx + c_0 \int_{\T^n} |D(u^1-u^2)|^2 m^1\,dx.
\]
Combine the two above inequalities and use (A2), (A3) to imply 
\begin{align*}
&\frac{d}{dt} \int_{\T^n} (u^1-u^2) (m^2-m^1) \,dx\\
 \leq & - \int_{\T^n} \rho*(\rho*(m^1-m^2)) (m^1-m^2)\,dx + c_0  \int_{\T^n} |D(u^1-u^2)|^2 (m^1+m^2)\,dx\\
 \leq  &- \int_{\T^n \times \T^n} \rho(x-y) (\rho*(m^1-m^2))(y) (m^1-m^2)(x)\,dydx + 2c_0 M \int_{\T^n} |D(u^1-u^2)|^2\,dx\\
 = &-\int_{\T^n} \left| \rho*(m^1-m^2)(y) \right|^2\,dy +  2c_0 M\int_{\T^n} |D(u^1-u^2)|^2\,dx.
 \end{align*}
Thus,
\begin{equation}\label{imp-1}
\frac{d}{dt} \int_{\T^n} (u^1-u^2) (m^2-m^1) \,dx \leq -\int_{\T^n} \left| \rho*(m^1-m^2) \right|^2\,dx +2c_0 M \int_{\T^n} |D(u^1-u^2)|^2\,dx.
\end{equation}

Next, we take the difference of first equations of \eqref{MFG1} and \eqref{MFG2}, multiply by $2(u^1-u^2)$ and integrate to imply
\begin{align*}
&\frac{d}{dt} \int_{\T^n} (u^1-u^2)^2\,dx\\
 = & - 2\int_{\T^n} |D(u^1-u^2)|^2\,dx - 2\int_{\T^n} (H(Du^1)-H(Du^2))(u^1-u^2)\,dx\\
& \qquad \qquad\qquad \qquad \qquad\qquad \qquad\qquad\qquad + 2\int_{\T^n} \rho*(\rho*(m^1-m^2)) (u^1-u^2)\,dx\\
 \leq & -2 \int_{\T^n} |D(u^1-u^2)|^2\,dx + 2c_0 \int_{\T^n} |D(u^1-u^2)| \cdot |u^1-u^2|\,dx\\
 &\qquad \qquad\qquad \qquad \qquad\qquad \qquad\qquad\qquad +2\int_{\T^n} (\rho*(m^1-m^2))( \rho *(u^1-u^2))\,dx\\
 \leq & -\int_{\T^n} |D(u^1-u^2)|^2\,dx +c_0^2 \int_{\T^n} (u^1-u^2)^2\,dx+4\int_{\T^n} \left| \rho*(m^1-m^2) \right|^2\,dx \\
  &\qquad \qquad\qquad \qquad \qquad\qquad \qquad\qquad\qquad \qquad \qquad\qquad
  +\frac{1}{4} \int_{\T^n}  \left| \rho*(u^1-u^2) \right|^2\,dx,
\end{align*}
where we use Cauchy-Schwarz's inequality in the last line.
Besides,
\[
\int_{\T^n}  \left| \rho*(u^1-u^2) \right|^2\,dx \leq \int_{\T^n}  (u^1-u^2)^2\,dx.
\]
Hence,
\begin{multline}\label{imp-2}
\frac{d}{dt} \int_{\T^n} (u^1-u^2)^2\,dx \leq -\int_{\T^n} |D(u^1-u^2)|^2\,dx
+\left(\frac{1}{4}+c_0^2\right) \int_{\T^n} (u^1-u^2)^2\,dx\\
+4\int_{\T^n} \left| \rho*(m^1-m^2) \right|^2\,dx.
\end{multline}

We continue by taking the difference of the second equations of \eqref{MFG1} and \eqref{MFG2}, multiply by $2(m^1-m^2)$
and integrate
\begin{align*}
&\frac{d}{dt} \int_{\T^n} -(m^1-m^2)^2\,dx\\
= & -2\int_{\T^n} |D(m^1-m^2)|^2\,dx - 2\int_{\T^n}   (m^1-m^2) D(m^1-m^2)\cdot DH(Du^1) \,dx\\
& \qquad \qquad\qquad \qquad \qquad\qquad \qquad
-2 \int_{\T^n} m^2 D(m^1-m^2)\cdot (DH(Du^1)-DH(Du^2))\,dx\\
\leq & -2\int_{\T^n} |D(m^1-m^2)|^2\,dx + 2c_0 \int_{\T^n} |D(m^1-m^2)|\cdot |m^1-m^2|\,dx\\
& \qquad \qquad\qquad \qquad \qquad\qquad \qquad\qquad\qquad
+2c_0 M \int_{\T^n} |D(m^1-m^2)|\cdot |D(u^1-u^2)|\,dx\\
\leq & -(2-c_0(1+M))\int_{\T^n} |D(m^1-m^2)|^2\,dx + c_0 \int_{\T^n} \left((m^1-m^2)^2+M|D(u^1-u^2)|^2\right)\,dx.
\end{align*}
Note that Cauchy-Schwarz's inequality is used in the last line of the above computation.
Note further that $\int_{\T^n} (m^1-m^2)\,dx =0$. 
Hence, Poincar\'e's inequality gives
\[
\int_{\T^n} |D(m^1-m^2)|^2\,dx \geq \int_{\T^n} (m^1-m^2)^2\,dx.
\]
Combine this with the previous computation, we arrive at
\begin{equation}\label{imp-3}
\frac{d}{dt} \int_{\T^n} -(m^1-m^2)^2\,dx \leq -(2-c_0(2+M))\int_{\T^n} (m^1-m^2)^2\,dx
+ c_0 M \int_{\T^n}|D(u^1-u^2)|^2\,dx.
\end{equation}

Define
\[
\varphi(t) = \int_{\T^n} \left((u^1-u^2)(m^2-m^1)+\frac{(u^1-u^2)^2}{4} - (m^1-m^2)^2 \right)\,dx.
\]
Multiply \eqref{imp-2} by $\frac{1}{4}$, combine the result with \eqref{imp-1} and \eqref{imp-3} to get that
\begin{multline*}
\varphi'(t) \leq -\left(\frac{1}{4}-3c_0 M \right) \int_{\T^n} |D(u^1-u^2)|^2\,dx +\left(\frac{1}{16}+\frac{c_0^2}{4} \right) \int_{\T^n} (u^1-u^2)^2\,dx\\
-(2-c_0(2+M))\int_{\T^n} (m^1-m^2)^2\,dx.
\end{multline*}
We use (A4) and \eqref{eq:c0} to get further that
\begin{equation}\label{imp-4}
\varphi'(t) \leq \left(\frac{1}{16}+\frac{1}{80} \right) \int_{\T^n} (u^1-u^2)^2\,dx- \frac{7}{4} \int_{\T^n} (m^1-m^2)^2\,dx
\leq \frac{1}{2} \varphi(t).
\end{equation}
Thus, $t\mapsto e^{-t/2} \varphi(t)$ is non-increasing on $[0,T]$. Note that
\[
\varphi(0)=\int_{\T^n} - (m^1-m^2)^2 \,dx \leq 0 \quad \text{and} \quad 
\varphi(T)=\int_{\T^n} \frac{(u^1-u^2)^2}{4}\,dx \geq 0,
\]
which imply that $\varphi \equiv 0$ and in fact $(u^1,m^1)=(u^2,m^2)$.
\end{proof}

\begin{rem}\label{rem:restrict}
Condition (A4) is quite restrictive as it requires that both $\|DH\|_{L^\infty(\R^n)}$ and $\|D^2H\|_{L^\infty(\R^n)}$ are small enough (smaller than $c_0$).
In particular, if $M=\max_{\T^n} m_T$ is sufficiently large, then Theorem \ref{thm:unique} gives the uniqueness result in a perturbative regime only.
This is of course not so satisfying.
To some extend, this is related to the result of Ambrose \cite{Am}.

The monotonicity of $\varphi(t)$ is interesting in its own right. See \cite[Section 6]{GV} for some related discussions.
\end{rem}

We provide next another uniqueness result, where the appearance of a constant drift is allowed.

\begin{thm} \label{thm:unique-new}
Let $b\in \R^n$ be a fixed vector, and $K:\R^n \to \R$ be a smooth function such that  {\rm (A1)--(A4)} hold with $K$ in place of $H$.
Define $H:\R^n \to \R$ as 
\[
H(p) = b\cdot p + K(p) \quad \text{for all $p\in \R^n$.}
\]
Then \eqref{MFG} has at most one pair of solution  $(u,m)\in C^2_1(\T^n \times [0,T]) \times C([0,T],L^2(\T^n))$.
\end{thm}

It is clear that this uniqueness result is stronger that that in Theorem \ref{thm:unique}.
We choose to present the two results separately to emphasize an important point that
there are some good cancelations corresponding to the constant drift term.

\begin{proof}
The proof is basically the same as that of Theorem \ref{thm:unique} except the fact that we need to handle the drift term $b\cdot p$ in a careful manner.
We cannot just use brute force bounds here.
Let us provide the computations related to these terms here.

The first term we need to take care of is
\begin{align*}
&- 2\int_{\T^n} (H(Du^1)-H(Du^2))(u^1-u^2)\,dx\\
=\, & - 2\int_{\T^n} (K(Du^1)-K(Du^2))(u^1-u^2)\,dx - 2\int_{\T^n} b\cdot D(u^1-u^2)(u^1-u^2)\,dx\\
\leq \ & 2 c_0 \int_{\T^n} |D(u^1-u^2)| \cdot |u^1-u^2|\,dx  - \int_{\T^n} b \cdot D((u^1-u^2)^2)\,dx\\
=\,  & 2 c_0 \int_{\T^n} |D(u^1-u^2)| \cdot |u^1-u^2|\,dx.
\end{align*}
The second term that we need to pay attention to is handled in the same way
\begin{align*}
 &- 2\int_{\T^n}   (m^1-m^2) D(m^1-m^2)\cdot DH(Du^1) \,dx\\
 = \, & - 2\int_{\T^n}   (m^1-m^2) D(m^1-m^2)\cdot DK(Du^1) \,dx- 2\int_{\T^n}   (m^1-m^2) D(m^1-m^2)\cdot b  \,dx\\\
 =\, & - 2\int_{\T^n}   (m^1-m^2) D(m^1-m^2)\cdot DK(Du^1) \,dx \\
 \leq \, &  2 c_0 \int_{\T^n} |D(m^1-m^2)| \cdot |m^1-m^2|\,dx.
\end{align*}
\end{proof}
\bibliographystyle{amsplain}
\providecommand{\bysame}{\leavevmode\hbox to3em{\hrulefill}\thinspace}
\providecommand{\MR}{\relax\ifhmode\unskip\space\fi MR }
\providecommand{\MRhref}[2]{%
  \href{http://www.ams.org/mathscinet-getitem?mr=#1}{#2}
}
\providecommand{\href}[2]{#2}

\end{document}